\documentclass[letterpaper, 10 pt, conference]{ieeeconf}  

\IEEEoverridecommandlockouts                              

\usepackage{todonotes}
\usepackage{latexsym,amssymb,amsmath, amsbsy,amsopn, amstext}
\usepackage{graphicx,epsfig,tikz,pgf,hyperref,cancel,color,float,ifpdf}
\usepackage{extarrows,mathrsfs}
\usepackage{amsmath,  amssymb,stfloats}
\usepackage{tikz,subcaption}
\usepackage[]{caption}
\usetikzlibrary{automata, positioning, arrows}
\usepackage[ruled,vlined,linesnumbered]{algorithm2e}
\usepackage{bm}

\usepackage{etoolbox}
\makeatletter
\patchcmd{\@makecaption}
  {\scshape}
  {}
  {}
  {}
\makeatother
\usepackage{makecell}

\graphicspath{{./figures/}}

\newtheorem{definition}{\bfseries Definition}
\newtheorem{proposition}{\bfseries Proposition}
\newtheorem{example}{\bfseries Example}
\newtheorem{theorem}{\bfseries Theorem}
\newtheorem{lemma}{\bfseries Lemma}

\newtheorem{remark}{\bfseries Remark}

\newtheorem{problem}{\bfseries Problem}

\def\x{\bm{x}}

\def\u{\bm{u}}

\def\f{\bm{f}}

\newcommand{\z}{\bm{z}}
\newcommand{\cc}{\mathbf{c}}
\newcommand{\G}{\mathbf{G}}
\newcommand{\A}{\mathbf{A}}
\newcommand{\bb}{\mathbf{b}}

\newcommand{\Zc}{\mathcal{Z}}

\newcommand{\mat}[1]{\begin{bmatrix} #1 \end{bmatrix}}

\allowdisplaybreaks[4]

\usepackage{xcolor}
\newif\ifdraft

\drafttrue

\title{\LARGE \bf
Reachability Analysis and Safety Verification of Neural Feedback Systems via Hybrid Zonotopes}

\author{Yuhao Zhang and Xiangru Xu
\thanks{Yuhao Zhang and Xiangru Xu are with the Department of Mechanical Engineering, University of Wisconsin-Madison,
        Madison, WI, USA. Email: 
        {\tt\small \{yuhao.zhang2,xiangru.xu\}@wisc.edu}.}%
}

\begin{document}
\maketitle
\begin{abstract}
Hybrid zonotopes generalize constrained zonotopes by introducing additional binary variables and possess some unique properties that make them convenient to represent nonconvex sets. 
This paper presents novel hybrid zonotope-based methods for the reachability analysis and safety verification of neural feedback systems. Algorithms are proposed to compute the input-output relationship of each layer of a feedforward neural network, as well as the exact reachable sets of neural feedback systems. In addition, a sufficient and necessary condition is formulated as a mixed-integer linear program to certify whether the trajectories of a neural feedback system can avoid unsafe regions. The proposed approach is shown to yield a formulation that provides the tightest convex relaxation for the reachable sets of the neural feedback system.  Complexity reduction techniques for the reachable sets are developed to balance the computation efficiency and approximation accuracy. Two numerical examples demonstrate the superior performance of the proposed approach compared to other existing methods.

\end{abstract}

\section{Introduction}\label{sec:intro}

Artificial neural networks have shown their extraordinary performance in many fields such as auto-driving systems \cite{grigorescu2020survey} and mobile robots \cite{pomerleau2012neural}. Implementation of neural networks in such controlled systems also raises safety concerns as even a small chance of failure may cause catastrophic consequences. Therefore, it's critical to find an efficient method to verify the safety properties of controlled systems with neural network components before real implementations. However, analyzing properties of neural networks is notoriously difficult due to their highly non-convex and nonlinear natures \cite{katz2017reluplex}.


Various methods have been proposed to perform reachability analysis and safety verification for the \emph{neural feedback systems} (i.e., feedback systems with neural network controllers) \cite{huang2019reachnn,ivanov2019verisig,yin2021stability,kochdumper2022open,dai2021lyapunov}. Based on quadratic constraints, a reachable set over-approximation method was proposed in \cite{Fazlyab2022Safety,hu2020reach} using semi-definite programming (SDP). A fast reachability method was introduced in \cite{everett2021reachability} by relaxing the SDP into linear programming (LP). Learning-based reachability methods were also developed in \cite{chakrabarty2020active,devonport2020data} for neural feedback systems with probabilistic guarantees on the correctness of the approximated reachable sets. Set-based methods were also proposed to compute the exact reachable sets of neural feedback systems using star sets \cite{tran2019safety} and constrained zonotopes \cite{zhang2022safety}. Despite their interesting results, these two methods can only deal with convex set representations which limit their usage for complex safety verification problems. Besides, the computation complexity increases rapidly for deep neural networks.


Recently, a new set representation named the \emph{hybrid zonotope} was introduced in \cite{bird2021hybrid}. Through the addition of binary generators, hybrid zonotope can represent non-convex sets with flat faces. And the reachability analysis based on hybrid zonotopes will lead to the formulation of mixed-integer linear programs (MILPs), for which many state-of-the-art solvers such as Gurobi \cite{gurobi} and learning-based solver MLOPT \cite{bertsimas2022online} can be utilized to accelerate the computation.

\begin{figure}[t]
    \centering
        \includegraphics[width=0.47\textwidth]{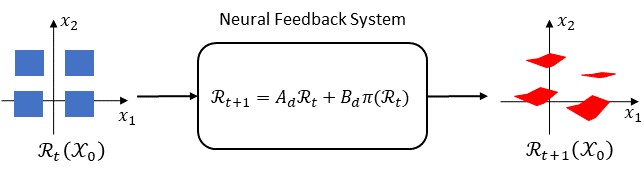}
     \caption{The neural feedback system is given as $\x{(t+1)} = A_d\x(t) + B_d \u(t)$ where the state feedback controller $\u(t) = \pi(\x(t))$ is a given $\ell$-layer FNN with the ReLU activation function. At each time step, the neural feedback system maps a hybrid zonotope as the input set to another hybrid zonotope as the output set. The initial set is $\mathcal{X}_0$ and the reachable set at time $t$ from $\mathcal{X}_0$ is $\mathcal{R}_t(\mathcal{X}_0)$.} 
      \label{fig:feedback-loop}
     \vspace{-0.3cm}
\end{figure}

In this work, we present hybrid zonotope-based methods for reachability analysis and safety verification of neural feedback systems with ReLU-activated feed-forward neural network (FNN) controllers (see Figure \ref{fig:feedback-loop}). 
The contributions of this paper are threefold: (i) For neural feedback systems with hybrid zonotopes as the input sets, a novel approach is presented to compute the nonconvex exact reachable sets represented as hybrid zonotopes; (ii) Based on the convex relaxation property of the computed reachable sets and the properties of hybrid zonotopes, heuristic reduction methods are proposed to reduce the complexity growth of the hybrid zonotope sets;  (iii) Using the computed reachable sets, an MILP-based condition is provided to verify the unsafe region avoidance of neural feedback systems, for which off-the-shelf solvers can be employed. 
The efficiency of the proposed methods is demonstrated through two numerical examples.

\section{Preliminaries \& Problem Statement}\label{sec:pre}


\subsection{Hybrid Zonotopes}\label{sec:zono}


\begin{definition}\label{def:sets}
Let $\Zc,\Zc_c,\Zc_h \subset\mathbb{R}^n$. $\Zc$ is a \emph{zonotope} if \eqref{equ:zono} holds \cite{mcmullen1971zonotopes}, $\Zc_c$ is a \emph{constrained zonotope} if \eqref{equ:czono} holds \cite{scott2016constrained}, and $\Zc_h$ is a \emph{hybrid zonotope} if \eqref{equ:hzono} holds \cite{bird2021hybrid}:
\begin{align}
    &\exists (\mathbf{c}, \mathbf{G}) \in  \mathbb{R}^{n}\times \mathbb{R}^{n \times n_{g}}: \!\Zc =\left\{\mathbf{G} \bm{\xi}+\mathbf{c} \;|\; \|\bm{\xi}\|_{\infty} \leq 1\right\}, \label{equ:zono}\\
    &\exists (\mathbf{c}, \mathbf{G},\A,\bb) \in  \mathbb{R}^{n}\times \mathbb{R}^{n \times n_{g}}\times \mathbb{R}^{n_{c} \times n_{g}} \times \mathbb{R}^{n_{c}}: \notag\\
    &  \; \quad\quad\quad\quad\quad\quad\Zc_c=\left\{\mathbf{G} \bm{\xi}+\mathbf{c} \;|\; \|\bm{\xi}\|_{\infty} \leq 1, \mathbf{A} \bm{\xi}=\mathbf{b}\right\},\label{equ:czono}\\
    &\exists (\mathbf{c}, \mathbf{G}^c,\G^b,\A^c,\A^b,\bb) \in  \mathbb{R}^{n}\!\times\! \mathbb{R}^{n \times n_{g}}\!\times\! \mathbb{R}^{n \times n_{b}} \!\times\! \mathbb{R}^{n_{c}\times{n_g}}\notag \\
    & \!\quad\quad\quad\quad\quad\quad\quad\quad\quad \quad\quad\times \mathbb{R}^{n_{c}\times{n_b}}\times \mathbb{R}^{n_{c}}: \label{equ:hzono}\\ \notag
    & \mathcal{Z}_{h}\!=\!\left\{\mat{\G^c \!\! & \!\!\G^b}\mat{\bm{\xi}^c\\ \bm{\xi}^b}+\cc \left |\!\! \begin{array}{c}
{\mat{\bm{\xi}^c\\ \bm{\xi}^b} \in \mathcal{B}_{\infty}^{n_{g}} \times\{-1,1\}^{n_{b}}}, \\
{\mat{\A^c \!\!& \!\!\A^b}\mat{\bm{\xi}^c\\ \bm{\xi}^b}=\bb}
\end{array}\right.\!\!\!\right\},
\end{align}
where $\mathcal{B}_{\infty}^{n_g}=\left\{\bm{x} \in \mathbb{R}^{n_g} \;|\;\|\bm x\|_{\infty} \leq 1\right\}$ is the unit hypercube in $\mathbb{R}^{n_{g}}$. The shorthand notations of the zonotope, constrained zonotope and hybrid zonotope  are given by $\Zc = Z\langle \cc, \G\rangle$, $\mathcal{Z}_c= CZ\langle \mathbf{c}, \mathbf{G}, \mathbf{A}, \mathbf{b}\rangle$, and $\mathcal{Z}_{h}= HZ\langle \mathbf{c}, \mathbf{G}^c, \mathbf{G}^b, \mathbf{A}^c, \mathbf{A}^b, \mathbf{b}\rangle$, respectively. 
\end{definition}

Note that a hybrid zonotope degenerates into a constrained zonotope when $n_b=0$, and a constrained zonotope degenerates into a zonotope when $n_c=0$. For a given hybrid zonotope $HZ\langle \mathbf{c}, \mathbf{G}^c, \mathbf{G}^b, \mathbf{A}^c, \mathbf{A}^b, \mathbf{b}\rangle$, the vector $\cc$ is called the \emph{center}, the columns of $\G^b$ are called the \emph{binary generators}, and the columns of $\G^c$ are called the \emph{continuous generators} (or simply \emph{generators} if binary generators are not present). For simplicity, we define the set $\mathcal{B}(\A^c,\A^b,\bb) = \{(\bm{\xi}^c,\bm{\xi}^b) \in \mathcal{B}^{n_g}_\infty \times \{-1,1\}^{n_b} \;|\; \A^c\bm{\xi}^c + \A^b\bm{\xi}^b = \bb \}$. We denote $\G[:,i]$ as the $i$-th column of a matrix $\G$. The complexity of a hybrid zonotope is described by its \emph{degrees-of freedom order} or simply \emph{order} $o_h = (n_g+n_b-n_c)/n$. 
The equivalence of a hybrid zonotope with a finite collection of constrained zonotopes is stated by the result below.

\begin{lemma}\label{lemma:hz-cz}
\cite[Theorem 5]{bird2021hybrid} The set $\mathcal{Z}_h\subset \mathbb{R}^n$ is a hybrid zonotope if and only if it is the union of a finite number of constrained zonotopes.
\end{lemma}

Similar to constrained zonotopes, hybrid zonotopes are closed under linear map and intersections.

\begin{lemma}\cite[Proposition  7]{bird2021hybrid}\label{lemma:set-op}
For any $\mathbf R \in \mathbb{R}^{m \times n}$, $\mathcal{Z}_{h}=HZ\langle \cc_{z},\G_{z}^{c}, \G_{z}^{b},  \A_{z}^{c}, \A_{z}^{b}, \bb_{z}\rangle \subset \mathbb{R}^{n}$, $\mathcal{Y}_{h}=HZ\langle \cc_y, \G_{y}^{c},$ $\G_{y}^{b}, \A_{y}^{c}, \A_{y}^{b}, \bb_{y}\rangle \subset \mathbb{R}^{n}$, and $\mathcal{H}_{-}=\{\bm x \in \mathbb{R}^{m} | \bm h^{T} \bm x \leq f\} \subset \mathbb{R}^{n}$ the following identities hold:
\begin{align*}
\mathbf{R}\Zc_h =HZ & \langle \mathbf{R}\mathbf{c}_{z},\mathbf{R}\mathbf{G}^c_{z}, \mathbf{R}\mathbf{G}^b_{z}, \mathbf{A}^c_{z}, \mathbf{A}^b_{z}, \mathbf{b}_{z}\rangle,\\
\mathcal{Z}_{h} \cap \mathcal{Y}_{h}=HZ & \langle\cc_{z},\mat{\G_{z}^{c} & \mathbf{0} },\mat{\G_{z}^{b} & \mathbf{0} }, \\
&\!\!\!\! \mat{\A_{z}^{c} & \mathbf{0} \\ \mathbf{0} & \A_{y}^{c} \\ \G_{z}^{c} & -\G_{y}^{c} } ,\mat{\A_{z}^{b} & \mathbf{0} \\\mathbf{0} & \A_{y}^{b} \\ \G_{z}^{b} & -\G_{y}^{b} },
\mat{\bb_{z} \\
\bb_{y} \\
\cc_{y}-\cc_{z} }\rangle,\\
\mathcal{Z}_{h} \cap\mathcal{H}_{-}=HZ & \langle\cc_{z},\mat{\G_{z}^{c} & \mathbf{0} }, \G_{z}^{b}, 
\mat{\A_{z}^{c} & \mathbf{0} \\ \bm{h}^{T} \G_{z}^{c} & \frac{d_{m}}{2} },\\
& \mat{ \A_{z}^{b} \\ \bm h^{T} \G_{z}^{b} },
\mat{ \bb_{z} \\ f-\bm{h}^{T} \cc_{z}-\frac{d_{m}}{2} }\rangle,
\end{align*}
where $d_{m}=\sum_{i=1}^{n_{g, z}}\left|\bm{h}^{T} \G_{z}^{c}[:,i]\right|+\sum_{i=1}^{n_{b, z}}\left|\bm{h}^{T} \G_{z}^{b}[:,i]\right|+f-\bm{h}^{T}  \cc_{z}$.
\end{lemma}

Hybrid zonotopes are also closed under union operation.
\begin{lemma}\label{lemma:hz-union}
\cite[Proposition 1]{bird2021unions} (Union)
For any $\mathcal{Z}_h=HZ\left\langle \cc_z, \G_z^c, \G_z^b, \A_z^c, \A_z^b, \bb_z\right\rangle \subset \mathbb{R}^n$ and $\mathcal{W}_h=HZ\left\langle \cc_w,\G_w^c, \G_w^b,  \A_w^c, \A_w^b, \bb_w\right\rangle \subset \mathbb{R}^n$, define the vectors $\hat{\G}^b \in \mathbb{R}^n, \hat{\cc} \in \mathbb{R}^n, \hat{\A}_z^b \in \mathbb{R}^{n_{c, z}}, \hat{\bb}_z \in \mathbb{R}^{n_{c, z}}, \hat{\A}_w^b \in \mathbb{R}^{n_{c, w}}$, and $\hat{\bb}_w \in \mathbb{R}^{n_{c, w}}$, such that
$\hat{\G}^b = \frac{(\G_w^b \mathbf{1}+\cc_z)-(\G_z^b \mathbf{1}+\cc_w)}{2},\hat{\A}_z^b = \frac{-\A_z^b \mathbf{1}-\bb_z }{2},\hat{\A}_w^b = \frac{\A_w^b \mathbf{1}+\bb_w }{2},
\hat{\cc} = \frac{(\G_w^b \mathbf{1}+\cc_z)+(\G_z^b \mathbf{1}+\cc_w)}{2},
\hat{\bb}_z = \frac{-\A_z^b \mathbf{1}+\bb_z }{2},\hat{\bb}_w = \frac{-\A_w^b \mathbf{1}+\bb_w }{2}.
$
Then the union of $\mathcal{Z}_h$ and $\mathcal{W}_h$ is a hybrid zonotope $\mathcal{Z}_h \cup$ $\mathcal{W}_h=HZ\left\langle \cc_u,\G_u^c, \G_u^b,  \A_u^c, \A_u^b, \bb_u\right\rangle \subset \mathbb{R}^n$ where
\begin{align*}
\G_u^c &=\mat{
\G_z^c & \G_w^c & \mathbf{0}
}, \G_u^b=\mat{
\G_z^b & \G_w^b & \hat{\G}^b
}, \\
\A_u^c &=\mat{
\A_z^c & \mathbf{0} & \mathbf{0} \\
\mathbf{0} & \A_w^c & \mathbf{0} \\
\multicolumn{2}{c}{\A_{3}^{c}} & I
}, \A_u^b=\mat{
\A_z^b & \mathbf{0} & \hat{\A}_z^b \\
\mathbf{0} & \A_w^b & \hat{\A}_w^b \\
& \A_3^b &
},\\
\A_3^c & =\mat{
I & \mathbf{0} \\
-I & \mathbf{0} \\
\mathbf{0} & I \\
\mathbf{0} & -I \\
\mathbf{0} & \mathbf{0} \\
\mathbf{0} & \mathbf{0} \\
\mathbf{0} & \mathbf{0} \\
\mathbf{0} & \mathbf{0}
}, \A_3^b=\mat{
\mathbf{0} & \mathbf{0} & \frac{1}{2} \mathbf{1} \\
\mathbf{0} & \mathbf{0} & \frac{1}{2} \mathbf{1} \\
\mathbf{0} & \mathbf{0} & -\frac{1}{2} \mathbf{1} \\
\mathbf{0} & \mathbf{0} & -\frac{1}{2} \mathbf{1} \\
\frac{1}{2} I & \mathbf{0} & \frac{1}{2} \mathbf{1} \\
-\frac{1}{2} I & \mathbf{0} & \frac{1}{2} \mathbf{1} \\
\mathbf{0} & \frac{1}{2} I & -\frac{1}{2} \mathbf{1} \\
\mathbf{0} & -\frac{1}{2} I & -\frac{1}{2} \mathbf{1}
},\\
\cc_u&=\hat{\cc}, 
\bb_u=\mat{
\hat{\bb}_z^T &
\hat{\bb}_w^T &
\bb_3^T
}^{T},\\ \bb_3&=\mat{
\frac{1}{2} \mathbf{1}^T &
\frac{1}{2} \mathbf{1}^T &
\frac{1}{2} \mathbf{1}^T &
\frac{1}{2} \mathbf{1}^T &
\mathbf{0}^T &
\mathbf{1}^T &
\mathbf{0}^T &
\mathbf{1}^T
}^{T}.
\end{align*}
\end{lemma}

The emptiness of a hybrid zonotope can be checked by solving an MILP.
\begin{lemma}\label{lemma:empty} \cite{bird2021hybrid}
Given $\mathcal{Z}_h = HZ\langle \mathbf{c}, \mathbf{G}^c, \mathbf{G}^b, \mathbf{A}^c, \mathbf{A}^b, \mathbf{b}\rangle$ $\subset \mathbb{R}^n$, $\mathcal{Z}_h \not = \emptyset$ if and only if  $\min\{||\bm{\xi}^c||_{\infty} \;|\; \A^c\bm{\xi}^c + \A^b\bm{\xi}^b  = \bb, \bm{\xi}^c\in \mathbb{R}^{n_g}, \bm{\xi}^b\in\{-1,1\}^{n_b}\} \leq 1.
$
\end{lemma}

\subsection{Problem Statement}

Consider a discrete-time linear system:
\begin{equation}\label{dt-sys}
    \x{(t+1)} = A_d \x(t) + B_d \u(t)
\end{equation}
where $\x(t)\in \mathbb{R}^n,\; \u(t)\in \mathbb{R}^m$ are the state and the control input. $A_d \in \mathbb{R}^{n \times n}$, $B_d \in \mathbb{R}^{n \times m}$ are the state matrix and the input matrix, respectively. 

We assume a state-feedback controller
\begin{equation}
\u(t) = \pi(\x(t)),\label{dt-input}    
\end{equation}
which is parameterized by an $\ell$-layer FNN  with the Rectified Linear Unit (ReLU) activation function. The closed-loop system with system model \eqref{dt-sys} and controller \eqref{dt-input} is denoted as:
\begin{align}\label{close-sys}
    \x{(t+1)} = \f_{cl}(\x(t)) \triangleq A_d\x(t) + B_d \pi(\x(t)).
\end{align}

For the closed-loop system \eqref{close-sys}, we denote $\mathcal{R}_t(\mathcal{X}_0)\triangleq\{\x(t)\in \mathbb{R}^n | \x(0) \in \mathcal{X}_0, \x{(k+1)} = \f_{cl}(\x(k)), k = 0,1,\dots, t-1\}$ the (forward) reachable set at time $t$ from a given set of initial conditions $\mathcal{X}_0 \subset \mathbb{R}^n$.

For the $\ell$-layer FNN controller, let $\mathbf W^{(k-1)}$ be the $k$-th layer weight matrix and $\mathbf v^{(k-1)}$ be the $k$-th layer bias vector, for $k=1,\dots,\ell$. Denote $\x^{(k)}$ as the neurons of the $k$-th layer, then, for $k=1,\dots,\ell-1$, we have
\begin{align}\label{equ:NN}
\x^{(k)}=ReLU(\mathbf W^{(k-1)}\x^{(k-1)}+ \mathbf v^{(k-1)})    
\end{align}
where $\x^{(0)} = \x(t)$ and  $ReLU(\x) =\max \{0, \x\}$.
Only the linear map is applied in the last layer, i.e., $\pi(\x(t)) = \x^{(\ell)} = \mathbf W^{(\ell-1)}\x^{(\ell-1)}+\mathbf v^{(\ell-1)}$.


We assume the initial set and the unsafe set for the closed-loop system \eqref{equ:reach-cl} are both represented by hybrid zonotopes. In this paper, we will investigate the following two problems.


\begin{problem} (Reachability analysis) \label{prob:1}
Given an initial set $\mathcal{X}_0$ that is represented as a hybrid zonotope, the parameters of the FNN controller $\pi$ and a time horizon $T\in \mathbb{Z}_{>0}$, compute the reachable set $\mathcal{R}_t(\mathcal{X}_0)$ for the closed-loop system \eqref{close-sys} where $t=1,\dots,T$.
\end{problem}

\begin{problem} (Safety verification) \label{prob:2}
Given unsafe set $\mathcal{O}$ represented by a hybrid zonotope, verify whether the state trajectories of the closed-loop system \eqref{close-sys} can avoid the unsafe region for $t=1,\dots,T$.
\end{problem}



\section{Exact Reachability analysis and Safety Verification}\label{sec:reach}

In this section, we consider Problem \ref{prob:1} and Problem \ref{prob:2} for the closed-loop system with an FNN controller as in \eqref{close-sys}.

\subsection{Output Analysis of Standalone FNN}

In this subsection, we will present an algorithm to compute the exact output set of a given FNN as in \eqref{equ:NN} with an input set represented as a hybrid zonotope.

From the definition of the FNN in \eqref{equ:NN}, the output of layer $k$ is the input of layer $k+1$, for $k=1,\dots,\ell-1$. Therefore, the output set of an FNN defined in can be derived layer by layer and we will focus on finding the input-output relationship for one layer. 
Using Lemma \ref{lemma:set-op}, we can pass an input set as a hybrid zonotope $\mathcal{Z}_h =HZ\langle \mathbf{c},$ $\mathbf{G}^c,\mathbf{G}^b, \mathbf{A}^c, \mathbf{A}^b, \mathbf{b}\rangle$ through a linear map as $\mathbf W\mathcal{Z}_h+ \mathbf v = HZ\langle \mathbf W \mathbf{c}+ \mathbf v,\mathbf W \mathbf{G}^c,\mathbf W \mathbf{G}^b, \mathbf{A}^c, \mathbf{A}^b, \mathbf{b}\rangle$. Thus, the only difficulty remaining is to find the output of a ReLU activation function for a hybrid zonotope. 
Inspired by the output analysis algorithm for FNN using the star sets representation in \cite{tran2019star}, we present Algorithm \ref{alg:exact} to compute the exact output set using hybrid zonotopes. 
Note that in Line 13-14, $\mathcal{H}^i_- =  \{\x \in \mathbb{R}^{n} \mid \mathbf{e}_i^T\x \leq 0\}$ and $\mathcal{H}^i_+ =  \{\x \in \mathbb{R}^{n} \mid \mathbf{e}_i^T\x \geq 0\}$ denote the half-spaces with $i$-th canonical vector $\mathbf{e}_i$, for $i=1,\dots,n$.
Algorithm \ref{alg:exact} reveals that when the input set $\mathcal{Z}_h$ to the FNN $\pi$ is a hybrid zonotope, the exact output of the FNN can also be represented as a hybrid zonotope.

\begin{algorithm}
\caption{Exact output analysis for one layer of FNN via hybrid zonotopes}\label{alg:exact}
\KwIn{weight matrix $\mathbf W$, bias vector $\mathbf v$, hybrid zonotope input sets $\mathcal{Z}_h$}
\KwOut{exact output set $\mathcal{R}$ as a hybrid zonotope}

\SetKwFunction{reachnn}{ReachNN}
\SetKwFunction{reachrelu}{ReachReLU}
\SetKwFunction{step}{StepReLU}
  
  \SetKwProg{Fn}{Function}{:}{}
  \Fn{$\mathcal{R}$ = \reachnn{$\mathcal{Z}_h$,$\mathbf W$,$\mathbf v$}}{
        $\mathcal{R} = \mathbf W \mathcal{Z}_h + \mathbf v$ \tcp*{linear map}
        $[lb\quad up]\leftarrow$ range of $\x$ in $\mathcal{I}$   \tcp*{MILP}
        $map = find(lb<0)$\\
        \For {$i$ in $map$} {
        $\mathcal{R} = \step (\mathcal{R},i,lb[i],up[i])$
        }
        \KwRet $\mathcal{R}$\
  }
  
   \SetKwProg{Fn}{Function}{:}{}
  \Fn{$\tilde{\mathcal{R}}$ = \step{${\mathcal{R}}$,$i$,$lb_i$,$up_i$}}{
        $\mathbf E_i = [\mathbf e_1\; \cdots\; \mathbf e_{i-1}\; \mathbf{0}\; \mathbf e_{i+1}\; \cdots\; \mathbf e_{n_I}]$\\
        \If{$up_i \leq 0$}{
          ${\mathcal{I}} = \mathbf E_i {\mathcal{R}}$ \tcp*{linear map}
          }
        \If{$lb_i<0\; \&\; up_i >0$}{
          $\mathcal{I}_+ = {\mathcal{R}}\cap \mathcal{H}^i_+$ \tcp*{Lemma \ref{lemma:set-op}}
          $\mathcal{I}_- = {\mathcal{R}}\cap \mathcal{H}^i_-$ \tcp*{Lemma \ref{lemma:set-op}}
          ${\mathcal{I}} = \mathcal{I}_+ \cup \mathbf E_i \mathcal{I}_-$ \tcp*{Lemma \ref{lemma:hz-union}}
        }
        
        
        \KwRet $\tilde{\mathcal{R}} = {\mathcal{I}}$\
  }

\end{algorithm}

\begin{algorithm}[!ht]
\caption{Compute $\G_1,\G_2$ in Theorem \ref{thm:sum}}\label{alg:G}
\KwIn{continuous generator matrix $\G^c\in \mathbb{R}^{n\times n_g}$ and binary generator matrix $\G^b\in \mathbb{R}^{n\times n_b}$ from hybrid zonotope $\mathcal{Z}_h$, $n^\pi_b$ - the number of binary generators of $\mathcal{Z}_h^\pi$ }
\KwOut{matrices $\G_1,\G_2$}

\SetKwFunction{ComputeG}{ComputeG}
         $\G_1 \longleftarrow \G^c$; $\G_2 \longleftarrow \G^b$;\\
         $ k  \longleftarrow \log_2(n_b^\pi+1) - \log_2(n_b+1)$;\\
         \Repeat{$k\leq 0$}{
         $\G_1 \longleftarrow [\G_1\; \mathbf{0}_{n\times 1}]$;\\
         $\G_1 \longleftarrow [\G_1\; \G_1]$; 
         $\G_2 \longleftarrow [\G_2\; \G_2]$;\\
         $m \longleftarrow 2*(\text{\# columns of } \G_1 + \text{\# columns of } \G_2)$;\\
         $\G_1 \longleftarrow [\G_1\; \mathbf{0}_{n\times m}]$; 
         $\G_2 \longleftarrow [\G_2\; \mathbf{0}_{n\times 1}]$;\\
                $k\longleftarrow k-1$;
                }
        \KwRet $\G_1,\G_2$\
\end{algorithm}

\subsection{Exact Reachable Set for Neural Feedback System} \label{sec:lin}

Next, we consider the reachability analysis for the closed-loop system \eqref{close-sys}. Recall that $\f_{cl}(\x) = A_d\x + B_d \pi (\x)$.
Note that a conservative over-approximation of the exact reachable set can be obtained by trivially adding the two terms of $\f_{cl}$ with the Minkowski sum.  
The following theorem provides the \emph{exact} form of $ \f_{cl}(\mathcal{Z}_h) = \{\f_{cl}(\x)| \x\in \mathcal{Z}_h\}$ for a given hybrid zonotope $\mathcal{Z}_h$.

\begin{theorem}\label{thm:sum}
Given any hybrid zonotope $\mathcal{Z}_h = HZ\langle \cc,$ $\G^c,\G^b,\A^c,\A^b,\bb\rangle \subset\mathbb{R}^n$ where $\G^c\in\mathbb{R}^{n\times {n_g}}$, $\G^b\in\mathbb{R}^{n\times {n_b}}$, $\A^c\in\mathbb{R}^{n_c\times {n_g}}$ and $\A^b\in\mathbb{R}^{n_c\times {n_b}}$, let $\pi(\mathcal{Z}_h)=  HZ\langle \cc_\pi,\G^c_\pi,\G^b_\pi,\A^c_\pi,\A^b_\pi,\bb_\pi\rangle\triangleq\mathcal{Z}_h^{\pi}$ be the computed output set using Algorithm \ref{alg:exact}. Then,  
$
    \bm{f}_{cl}(\mathcal{Z}_h) = HZ\langle \cc_{cl},\G^c_{cl},\G^b_{cl},\A^c_{cl},\A^b_{cl},\bb_{cl}\rangle \triangleq {\mathcal{Z}}_{h}^{cl},
$
where
\begin{align*}
\G^c_{cl} &= A_d \G_1+B_d \G^c_\pi,\; \G^b_{cl} = A_d \G_2+B_d \G^b_\pi,\\
\cc_{cl} &= A_d (\cc+(\frac{n_b^\pi+1}{n_b+1}-1)\G^b\mathbf{1}) + B_d \cc_\pi,\;\\
\A^c_{cl} & = \A^c_{\pi},\;\A^b_{cl} = \A^b_{\pi},\;\bb_{cl}=\bb_\pi,
\end{align*}
matrices $\G_1$ and $\G_2$ are given by Algorithm \ref{alg:G}, $n_b$ and $n^\pi_b$ are the numbers of binary generators of $\mathcal{Z}_h$ and $\mathcal{Z}_h^{\pi}$, respectively.
\end{theorem}

\begin{proof}
We will firstly show that $f_{cl}(\mathcal{Z}_h) \subseteq {\mathcal{Z}}_{h}^{cl}$. Let $\x$ be any element of set $\mathcal{Z}_h$, i.e., $\x\in \mathcal{Z}_h$. Clearly, there exists $(\bm{\xi}_1^c,\bm{\xi}_1^b) \in \mathcal{B}(\A^c,\A^b,\bb)$, such that $\x = \cc+\G^c\bm{\xi}_1^c+\G^b\bm{\xi}_1^b$. 
From Line 13-15 in Algorithm \ref{alg:exact} and using Lemma \ref{lemma:set-op} and Lemma \ref{lemma:hz-union}, it can be observed that $\G^c_\pi$ has the same structure as $\G_1$ and $\G^b_\pi$ has the same structure as $\G_2$. 
Since $\pi(\x) \in \pi(\mathcal{Z}_h) = \mathcal{Z}_h^\pi$, there must exist $\bm{\xi}^c = \mat{ {\bm{\xi}_1^c}\\ {\bm{\xi}_2^c}}$ and $\bm{\xi}^b = \mat{ {\bm{\xi}_1^b}\\ {\bm{\xi}_2^b}}$ such that $(\bm{\xi}^c,\bm{\xi}^b) \in \mathcal{B}(\A^c_\pi,\A^b_\pi,\bb_\pi)$ and $\pi(\x) = \cc_\pi+\G^c_\pi\bm{\xi}^c+\G^b_\pi\bm{\xi}^b$. According to the switch rule between the two sets in the union operation in \cite{bird2021unions}, we have $\bm{\xi}^c_2 = \bm{0}$ and $\bm{\xi}^b_2 = \mat{\bm{-1}^T & 1 & \cdots & \bm{-1}^T & 1}^T$. Thus, we get $\G_1 \bm{\xi}^c = \G^c \bm{\xi}^c_1$ and $\G_2 \bm{\xi}^b = \G^b \bm{\xi}^b_1 - (\frac{n_b^\pi+1}{n_b+1}-1)\G^b \bm{1}$. Therefore, $\f_{cl}(\x) = A_d \x + B_d\pi(\x) = A_d (\cc+\G^c\bm{\xi}_1^c+\G^b\bm{\xi}_1^b) + B_d (\cc_\pi+\G^c_\pi\bm{\xi}^c+\G^b_\pi\bm{\xi}^b) = (A_d \cc + B_d \cc_\pi + (\frac{n_b^\pi+1}{n_b+1}-1)A_d\G^b\mathbf{1}) + (A_d \G_1 + B_d \G^c_\pi)\bm{\xi}^c + (A_d \G_2 + B_d \G^b_\pi)\bm{\xi}^b = \cc_{cl} + \G^c_{cl}\bm{\xi}^c + \G^b_{cl}\bm{\xi}^b$, where $(\bm{\xi}^c,\bm{\xi}^b) \in \mathcal{B}(\A^c_\pi,\A^b_\pi,\bb_\pi) = \mathcal{B}(\A^c_{cl},\A^b_{cl},\bb_{cl})$. Thus, we have $\f_{cl}(\x)\in {\mathcal{Z}}_{h}^{cl}$. As $\x$ is arbitrarily chosen in $\mathcal{Z}_h$, we get that $\f_{cl}(\mathcal{Z}_h) \subseteq {\mathcal{Z}}_{h}^{cl}$.

Next, we prove that ${\mathcal{Z}}_{h}^{cl}\subseteq \f_{cl}(\mathcal{Z}_h) $ holds. For any $\z \in {\mathcal{Z}}_{h}^{cl}$, there exists $(\bm{\xi}^c,\bm{\xi}^b) \in \mathcal{B}(\A^c_{cl},\A^b_{cl},\bb_{cl})$ such that $\z = \cc_{cl} + \G^c_{cl}\bm{\xi}^c + \G^b_{cl}\bm{\xi}^b = (A_d \cc + B_d \cc_\pi + (\frac{n_b^\pi+1}{n_b+1}-1)A_d\G^b\mathbf{1}) + (A_d \G_1 + B_d \G^c_\pi)\bm{\xi}^c + (A_d \G_2 + B_d \G^b_\pi)\bm{\xi}^b$. Partition $\bm{\xi}^c$ and $\bm{\xi}^b$ into $\mat{ {\bm{\xi}_1^c}\\ {\bm{\xi}_2^c}}$ and $\bm{\xi}^b = \mat{ {\bm{\xi}_1^b}\\ {\bm{\xi}_2^b}}$. It follows that $(\bm{\xi}_1^c,\bm{\xi}_1^b) \in \mathcal{B}(\A^c,\A^b,\bb)$ and $ \z = A_d \cc + B_d \cc_\pi + (\cc+(\frac{n_b^\pi+1}{n_b+1}-1)A_d\G^b\mathbf{1} + (A_d \G_1 + B_d \G^c_\pi)\mat{ {\bm{\xi}_1^c}\\ {\bm{\xi}_2^c}} + (A_d \G_2 + B_d \G^b_\pi)\mat{ {\bm{\xi}_1^b}\\ {\bm{\xi}_2^b}} = A_d (\cc+\G^c\bm{\xi}_1^c+\G^b\bm{\xi}_1^b) + B_d (\cc_\pi+\G^c_\pi\bm{\xi}^c+\G^b_\pi\bm{\xi}^b)$. Let $\x = \cc+\G^c\bm{\xi}_1^c+\G^b\bm{\xi}_1^b$, we get $\x\in\mathcal{Z}_h$ and $\pi(\x) = \cc_\pi+\G^c_\pi\bm{\xi}^c+\G^b_\pi\bm{\xi}^b$. Then, $\z = A_d\x+B_d\pi(\x) = \f_{cl}(\x)\in\f_{cl}(\mathcal{Z}_h)$. Since $\z$ is arbitrary, we have ${\mathcal{Z}}_{h}^{cl}\subseteq \f_{cl}(\mathcal{Z}_h)$. In conclusion, we have $\f_{cl}(\mathcal{Z}_h) = {\mathcal{Z}}_{h}^{cl}$.
\end{proof}

Based on Theorem \ref{thm:sum}, the exact reachable sets of closed-loop system \eqref{close-sys} can be computed as follows:
 \begin{align}
\mathcal{R}_0 = \mathcal{X}_0,\;\mathcal{R}_t = \f_{cl}(\mathcal{R}_{t-1}),\; t=1,\dots,T.\label{equ:reach-cl}
\end{align}

The reachable sets computed by \eqref{equ:reach-cl} are \emph{exact} as long as the initial set can be represented by a hybrid zonotope. The price of accuracy, however, is that the complexity order (i.e., the numbers of continuous and binary generators - $n_g$ and $n_b$) of the hybrid zonotope reachable sets will grow exponentially. If $n_\pi$ is the total number of neurons in $\pi$, then, in the worst case, $n_b$ will increase in the order of $2^{n_\pi}-1$ and $n_g$ will increase in the order of $4^{n_\pi}-1$. Thus, complexity reduction techniques are needed to reduce the computation burden, which will be introduced in the next section. 

\begin{remark}
In our prior work \cite{zhang2022safety}, a method based on constrained zonotopes was proposed to compute exact reachable sets of neural feedback systems. Different from the exact reachability analysis in this section, the input set considered in \cite{zhang2022safety} is limited to a single constrained zonotope, which is unable to represent \emph{non-convex sets} as the hybrid zonotope does.  
Although one may apply Lemma \ref{lemma:hz-union} to convert the unions of constrained zonotopes into hybrid zonotopes, this will result in a set with a larger complexity order as it will take much more union operations than Algorithm \ref{alg:exact} of this work. Numerical comparisons of these two methods will be demonstrated by examples in Section \ref{sec:sim}.
\end{remark}


\begin{remark}
Although only linear feedback systems are considered in this work, the proposed approach can be readily extended to general nonlinear feedback systems by abstracting nonlinear dynamics  with a set of optimally tight piecewise linear bounds as in \cite{sidrane2022overt}.
\end{remark}

\subsection{Safety Verification}
Denote the exact reachable set from initial set $\mathcal{X}_0$ at time $t$ computed by \eqref{equ:reach-cl} be $\mathcal{R}_t(\mathcal{X}_0) = HZ\langle \cc_t,\G^c_t,\G^b_t,\A^c_t,\A^b_t,\bb_t\rangle$ for $t = 1,\dots,T$. Assume the unsafe region is represented by a hybrid zonotope $\mathcal{O} = HZ\langle \cc_o,\G^c_o,\G^b_o,\A^c_o,\A^b_o,\bb_o\rangle$. The following result provides a sufficient and necessary condition on the safety verification of the closed-loop system \eqref{close-sys}.

\begin{proposition} \label{prop:verify}
Given the reachable sets $\mathcal{R}_1,\dots,\mathcal{R}_T$ and unsafe set $\mathcal{O}$ defined above, the state trajectories of the closed-loop system \eqref{close-sys} will not enter the unsafe region if and only if the following condition is satisfied for $t \in \{1,\dots,T\}$:
\begin{align} \label{equ:emp-check}
\!\!\min &\left\{ ||\bm{\xi}^c||_{\infty} \left | \mat{\A_t^c \!\!& \!\! \mathbf{0}\\ \mathbf{0} \!\!& \!\! \A_o^c \\ \G_t^c \!\!& \!\! -\G_o^c}\bm{\xi}^c  + \mat{\A_t^b \!\!& \!\! \mathbf{0}\\ \mathbf{0} \!\!& \!\! \A_o^b \\ \G_t^b \!\!& \!\! -\G_o^b}\bm{\xi}^b\right.\right.\notag\\ 
&\left.=  \mat{\bb_t\\ \bb_o\\ \cc_o - \cc_t},\bm{\xi}^c\in\mathbb{R}^{n_{g,t}}, \bm{\xi}^b\in\{-1,1\}^{n_{b,t}} \right\} > 1.
\end{align}
\end{proposition}

Avoiding unsafe regions can be equivalently expressed as none of the reachable sets intersect with the unsafe set. Since Proposition \ref{prop:verify} is a straight-forward application of Lemma \ref{lemma:set-op} and Lemma \ref{lemma:empty}, the proof is omitted due to space limitation.

\begin{remark}\label{remark:prop1}
The safety verification problem is formulated as $T$ MILPs \eqref{equ:emp-check} with $n_{g,t}$ continuous variables and $n_{b,t}$ binary variables. Although MILPs are well known to be NP-hard problems in general, 
some common commercial MILP solvers such as Gurobi \cite{gurobi} have shown promising performance in both average solving time and wide ranges of solvable problems. Recently, learning-based MILP solvers were proposed in \cite{bertsimas2022online} that can significantly alleviate the computational burden . The fast development of these MILP solvers enables us to incorporate these off-the-shelf tools into our verification problem.
\end{remark}

\section{Complexity Reduction} \label{sec:reduction}
Due to the intersection and union operation in Algorithm \ref{alg:exact}, both the number of continuous generators and the number of binary generators will increase fast. As mentioned in Section \ref{sec:lin}, in the worst case, these two numbers will grow exponentially which makes Theorem \ref{thm:sum} computationally heavy. In this section, we will introduce two order reduction techniques that can provide over-approximated reachable sets with fewer continuous and binary generators.

\subsection{Reducing Number of Binary Generators}

Given a hybrid zonotope, it's possible that the set can be represented by another hybrid zonotope with fewer binary generators. A rigorous approach is proposed in \cite{bird2021hybrid} to remove the redundant binary generators by exploring the independent feasible solutions for the binary variable $\bm{\xi}^b$. Although this approach can reduce the complexity of the hybrid zonotope representation, one major limitation is that it can not further reduce the binary generators without altering the set. In this subsection, however, we explore the relationship between the union and convex hull operations of hybrid zonotopes, and provide a novel method to reduce the number of binary generators while guaranteeing an over-approximation.

Let's first consider two constrained zonotopes $\mathcal{Z}_c = CZ\langle \cc_z,\G_z,\A_z,\bb_z \rangle\subset\mathbb{R}^n$ and $\mathcal{W}_c = CZ\langle \cc_w,\G_w,\A_w,\bb_w\rangle\subset\mathbb{R}^n$. It's obvious that $\mathcal{Z}_c$ and $\mathcal{W}_c$ are also equivalent to two degenerated hybrid zonotopes: $HZ\langle \cc_z,\G_z,\emptyset,\A_z,\emptyset,\bb_z\rangle$ and $ HZ\langle \cc_w,\G_w,\emptyset,\A_w,\emptyset,\bb_w\rangle$.

For any set $\mathcal{X}\subset \mathcal{R}^n$, we denote the convex hull of $\mathcal{X}$ as $conv(\mathcal{X})$ \cite{boyd2004convex}. According to Theorem 5 in \cite{raghuraman2022set}, we can compute the convex hull of $\mathcal{Z} \cup \mathcal{W}$ as a constrained zonotope $\mathcal{C}_c = conv(\mathcal{Z}_c \cup \mathcal{W}_c) = CZ\langle \cc_{ch},\G_{ch},\A_{ch},\bb_{ch}\rangle$ where
\begin{align}\label{equ:conv1}
&{\G}_{ch}=\mat{\G_z & \G_w & \frac{\cc_z-\cc_w}{2} & \bm{0}}, \; {\cc}_{ch}=\frac{{\cc}_{z}+{\cc}_{w}}{2}, \\
&{\A}_{ch}= \mat{\A_z & \bm{0} & -\frac{\bb_z}{2} & \bm{0} \\ \bm{0} & \A_w & \frac{\bb_w}{2} & \bm{0}\\ \A_{3,1} & \A_{3,2} & \A_{3,1}& I}, \; {\bb}_{ch}=\mat{\frac{1}{2}\bb_z\\\frac{1}{2}\bb_w\\-\frac{1}{2}\bm{1}} \text {, } \label{equ:conv2}\\ 
&{\A}_{3,1}=\mat{I\\-I\\\bm{0}\\\bm{0}}, {\A}_{3,2}=\mat{\bm{0}\\\bm{0}\\I\\-I}, \; {\A}_{3,0}= \mat{-\frac{1}{2}\bm{1} \\ -\frac{1}{2}\bm{1} \\ \frac{1}{2}\bm{1}\\ \frac{1}{2}\bm{1}}. \label{equ:conv3}
\end{align}

According to Lemma \ref{lemma:hz-union}, we can compute the union of $\mathcal{Z}_c$ and $\mathcal{W}_c$ as a hybrid zonotope: $\mathcal{U}_h = \mathcal{Z}_c \cup \mathcal{W}_c = HZ\langle \cc_u,\G^c_u,\G^b_u,\A_u^c,\A_u^b,\bb_u\rangle$ where
\begin{align}\label{equ:union1}
\G_{u}^{c} &\!=\!\mat{
\G_{z} \!&\! \G_{w} \!&\! \mathbf{0}
}, 
\G_{u}^{b}\!=\!
\frac{\cc_z-\cc_w}{2}
, \cc_{u}=\frac{\cc_z+\cc_w}{2} \\ \label{equ:union2}
\A_{u}^{c} &\!=\!\mat{
\A_{z} \!&\! \mathbf{0} \!&\! \mathbf{0} \\
\mathbf{0} \!&\! \A_{w} \!&\! \mathbf{0} \\
\multicolumn{2}{c}{\A_{3}^{c}} \!&\! I
}, 
\A_{u}^{b}\!=\!\mat{
 \frac{-\bb_z}{2} \\
 \frac{\bb_w}{2} \\
\A_{3}^{b}
}, \bb_{u}\!=\!\mat{
\frac{\bb_z}{2} \\
\frac{\bb_w}{2} \\
\bb_{3}
},\\
\A_{3}^{c} &\!=\!\mat{
I \!\!&\!\! \mathbf{0} \\
-I \!\!&\!\! \mathbf{0} \\
\mathbf{0} \!\!&\!\! I \\
\mathbf{0} \!\!&\!\! -I
}, \A_{3}^{b}\!=\!\mat{
\mathbf{0} \!\!&\!\! \mathbf{0} \!\!&\!\!\frac{-1}{2}\mathbf{1} \\
\mathbf{0} \!\!&\!\! \mathbf{0} \!\!&\!\! \frac{-1}{2}\mathbf{1} \\
\mathbf{0} \!\!&\!\! \mathbf{0} \!\!&\!\! \frac{1}{2} \mathbf{1} \\
\mathbf{0} \!\!&\!\! \mathbf{0} \!\!&\!\! \frac{1}{2} \mathbf{1}
}, 
\bb_{3}\!=\!\mat{
\frac{-1}{2}\mathbf{1} \\
\frac{-1}{2}\mathbf{1} \\
\frac{-1}{2}\mathbf{1} \\
\frac{-1}{2}\mathbf{1}
}\!.\label{equ:union3}
\end{align}

The relationship between the union and convex hull of two constrained zonotopes is summarized below.

\begin{lemma}\label{lemma:CH_relax}
Consider two constrained zonotopes $\mathcal{Z}_c, \mathcal{W}_c\subset\mathbb{R}^n$, and let $\mathcal{U}_h=\mathcal{Z}_c\cup \mathcal{W}_c$ be a hybrid zonotope computed as in \eqref{equ:union1}-\eqref{equ:union3} and $\mathcal{C}_c=conv(\mathcal{Z}_c\cup \mathcal{W}_c)$ be a constrained zonotope computed as in \eqref{equ:conv1}-\eqref{equ:conv3}. 
If the binary variable constraint in $\mathcal{U}_h$ is relaxed to continuous variable constraint, i.e., replace $\xi^b\in\{-1,1\}$ with $\xi^b \in [-1,1]$, to get a relaxed constrained zonotope $\mathcal{U}_c$, then $\mathcal{U}_c$ is equivalent to $\mathcal{C}_c$.
\end{lemma}
\begin{proof}
It's easy to check that $\G_{ch} = [\G_z \;\; \G_w \;\; \G^b_u \;\; \bm{0}]$ and $\A_{ch} = \mat{\A_{z} & \mathbf{0} & -\frac{\bb_z}{2} & \mathbf{0} \\
\mathbf{0} & \A_{w} & \frac{\bb_w}{2} & \mathbf{0} \\
& \A_{3}^{c} & \bb_{3}& I}$.
Therefore, we have $\mathcal{C}_c = CZ\langle \cc_{ch},\G_{ch},$ $\A_{ch},\bb_{ch}\rangle = CZ\langle \cc_u,\mat{\G_u^c & \G_u^b}, \mat{\A_u^c & \A_u^b},\bb\rangle = \mathcal{U}_c$.
\end{proof}

The following theorem extends Lemma \ref{lemma:CH_relax} to the reachable sets computed by \eqref{equ:reach-cl}.

\begin{theorem}\label{thm-CH}
Given any hybrid zonotope $\mathcal{Z}_h = \langle \cc,\G^c,\G^b,\A^c,\A^b,\bb\rangle$ from the reachable set computation \eqref{equ:reach-cl}, the convex hull of $\mathcal{Z}_h$ can be constructed as the constrained zonotope $\mathcal{Z}_c = CZ\langle \cc,\mat{\G^c & \G^b},\mat{\A^c & \A^b},\bb\rangle$, i.e., $\mathcal{Z}_c = conv(\mathcal{Z}_h)$.
\end{theorem}
\begin{proof}
From Line 15 in Algorithm \ref{alg:exact} and Lemma \ref{lemma:hz-cz}, we know that $\mathcal{Z}_h$ can be represented by the union of a finite number of constrained zonotopes. Without loss of generality, assume $\mathcal{Z}_h = \mathcal{Z}_{c,1}\cup \mathcal{Z}_{c,2} \cup \cdots \cup \mathcal{Z}_{c,N}$ with $\Zc_{c,i},,i=1,\dots,N$ being constrained zonotopes. It can be observed that using Lemma \ref{lemma:CH_relax}, we can eliminate one binary variable each time by replacing the union of two constrained zonotopes with their convex hull. Based on properties of convex hull, by repeating the same procedure for $N-1$ times, we can get $\mathcal{Z}_c = CZ\langle \cc,\mat{\G^c & \G^b},\mat{\A^c & \A^b},\bb\rangle = conv(\cdots conv(\mathcal{Z}_{c,1}\cup \mathcal{Z}_{c,2})\cdots \cup \mathcal{Z}_{c,N}) = conv(\mathcal{Z}_h)$.
\end{proof}

\begin{remark}
Note that Theorem \ref{thm-CH} is not true for an arbitrary hybrid zonotope. Although for any hybrid zonotope, relaxing the binary constraints into linear constraints leads to an over-approximation of the hybrid zonotope, it's not guaranteed to be the tightest convex relaxation. However, Theorem \ref{thm-CH} shows that our reachable set formulation computed by \eqref{equ:reach-cl} can provide the \emph{tightest} convex relaxation of the neural feedback systems with ReLU-activated FNN controllers. This property is similar to the \emph{ideal formulation} for MILPs in \cite{anderson2020strong}. Finding the tightest convex relaxation for a general hybrid zonotope is still an open problem and will be explored in our future work.
\end{remark}

\begin{example}\label{exp}
Consider a hybrid zonotope reachable set computed by \eqref{equ:reach-cl}: $\mathcal{Z}_h = HZ\langle \cc,\G^c,\G^b,\A^c,\A^b,\bb\rangle$ where
\begin{align*}
     &\cc\! =\! \mat{0.25\\2.25}, \G^c\! = \!\mat{-1&1&0&0&-0.5&1& 0&0\\-1&-1&0&0&-1&-0.5 & 0 & 0},\\
     &\G^b = \mat{-0.75\\-0.75},\A^c = \mat{1&0&1&0&0&0&0&0\\0&1&0&1&0&0&0&0\\0&0&0&0&1&0&1&0\\0&0&0&0&0&1&0&1},\\
     &\A^b = \mat{1&1&-1&-1}^T, \bb = \mat{1&1&1&1}^T.
\end{align*}
Using Theorem \ref{thm-CH} to relax all the binary constraints, then the constrained zonotope $\mathcal{Z}_c = CZ\langle \cc,\mat{\G^c & \G^b},\mat{\A^c & \A^b},\bb\rangle$ is the convex hull of $\mathcal{Z}_h$ as depicted in Figure \ref{fig:exp}, i.e. $\mathcal{Z}_c = conv(\mathcal{Z}_h)$.
\end{example}

\begin{figure}[t]
    \centering
        \includegraphics[width=0.3\textwidth]{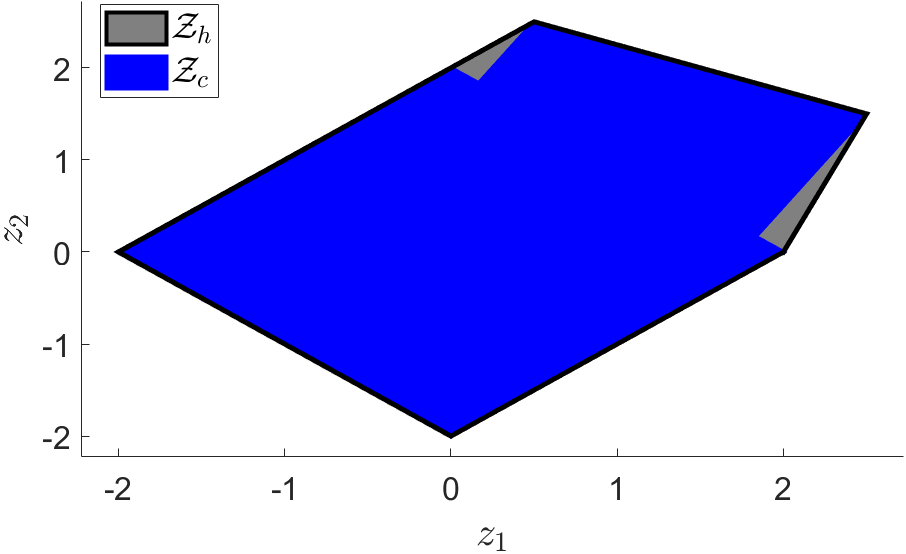}
     \caption{Hybrid zonotope $\mathcal{Z}_h$ (blue) and constrained zonotope $\mathcal{Z}_c$ (gray) given in Example \ref{exp}. $\mathcal{Z}_c$ constructed by relaxing all the binary constraints becomes the convex hull of $\mathcal{Z}_h$.}
      \label{fig:exp}
      \vspace{-0.3cm}
\end{figure}

Using Theorem \ref{thm-CH}, we can reduce the desired number of binary generators of a hybrid zonotope by replacing them with the same number of continuous generators. For example, given a hybrid zonotope $\mathcal{Z}_h$ with $n_g$ continuous generators and $n_b$ binary generators, we can reduce $\hat n_b$ binary generators and get an over-approximated hybrid zonotope $\hat{\mathcal{Z}}_h$ with $n_g+\hat n_b$ continuous generators and $n_b-\hat n_b$ binary generators. When $\hat n_b = n_b$, $\hat{\mathcal{Z}}_h$ becomes a constrained zonotope which is also the convex hull of $\mathcal{Z}_h$.

\subsection{Reducing Number of Continuous Generators}

In this subsection, we introduce two methods to reduce the number of continuous generators. For a zonotope, generator reduction can be done by identifying parallel generators and combining parallel generators through addition \cite{althoff2015computing}. The same approach can be applied to a constrained zonotope $\Zc_c = CZ\langle \cc,\G,\A,\bb\rangle$ if the lifted zonotope 
$
\Zc^+ = Z\langle \mat{\cc\\ \bb}, \mat{\G\\ \A}\rangle =  Z\langle \cc^+,\G^+ \rangle
$
has parallel generators $\G^+[:,i]\; ||\; \G^+[:,j]$ \cite{scott2016constrained}. In this case, the parallel generators can be similarly reduced by
simply combining parallel generators through addition $\G^+[:,i] + \G^+[:,j]$ in the lifted zonotope which is then transformed back to a reduced constrained zonotope with fewer generators.

This \emph{lift-then-reduce} strategy can also be extended to hybrid zonotopes. The following proposition is inspired by similar results for constrained zonotopes in \cite{scott2016constrained}.

\begin{proposition}\label{prop:HZ_reduce}
Consider a hybrid zonotope $\Zc_h = HZ\langle \cc,\G^c,\G^b,\A^c,\A^b,\bb\rangle \subset \mathbb{R}^n$ and a partition $\mat{\A^c & \A^b & \bb} = \mat{\A^c_1 & \A^b_1 & \bb_1\\ \A^c_2 & \A^b_2& \bb_2}$. For every $\bm{z}\in \mathbb{R}^n$, $\bm{z} \in \Zc_h$ if and only if
$
\mat{\bm{z}\\ \bm{0}} \in \Zc_h^+ \triangleq HZ\left\langle \mat{\cc \\ -\bb_1}, \mat{\G^c\\ \A^c_1}, \mat{\G^b \\ \A^b_1}, \A^c_2, \A^b_2, \bb_2\right\rangle.
$
\end{proposition}

\begin{proof}
We have $\bm{z} \in \Zc_h$ iff $\exists (\bm{\xi}^c ,\bm{\xi}^b) \in \mathcal{B}(\A^c,\A^b,\bb)$ such that $\bm{z} =\cc + \G^c \bm{\xi}^c + \G^b \bm{\xi}^b$. And clearly, the latter condition holds iff $\exists (\bm{\xi}^c ,\bm{\xi}^b) \in \mathcal{B}(\A^c_2,\A^b_2,\bb_2)$, such that $\mat{\bm{z}\\ \bm{0}} = \mat{\cc\\ -\bb_1} + \mat{\G^c \\ \A^c_1} \bm{\xi}^c + \mat{\G^b \\ \A^b_1} \bm{\xi}^b$.
\end{proof}

For a hybrid zonotope $\Zc_h = HZ\langle \cc,\G^c,\G^b,\A^c,\A^b,\bb\rangle$, we can form a lifted hybrid zonotope $\Zc_h^+$ by Proposition \ref{prop:HZ_reduce}:
\begin{equation}\label{equ:lift}
    \Zc_h^+ = HZ\left\langle \mat{\cc \\ -\bb}, \mat{\G^c\\ \A^c}, \mat{\G^b \\ \A^b}, \emptyset, \emptyset, \emptyset \right\rangle.
\end{equation}

It's obvious that this lifted hybrid zonotope $\Zc_h^+$ is equivalent to a union of lifted zonotopes with the same group of generators and shifted centers, i.e., $\Zc_h = \mathcal{Z}_1\cup \mathcal{Z}_2 \cup \cdots \cup \mathcal{Z}_{2^{n_b}}$, where $\Zc_i = Z\left\langle \mat{\cc+\G^b \bm{\xi}_i^b\\ -\bb+\A^b \bm{\xi}_i^b},\mat{\G^c\\ \A^c} \right\rangle$ and $\bm{\xi}_i^b \in \{-1,1\}^{n_b}$ for $i=1,\dots,2^{n_b}$. 

Therefore, if there exist parallel generators for any of the lifted zonotopes, all the other lifted zonotopes have the same set of parallel generators. We then combine the parallel generators for the lifted zonotopes and use Proposition \ref{prop:HZ_reduce} to transform the reduced lifted zonotopes back to a reduced hybrid zonotope with fewer continuous generators.


The approach described above can be used to remove the continuous generators based on the generator directions. In what follows, we provide another method to reduce a continuous generator and an equality constraint at the same time. The following proposition extends Proposition 5 in \cite{scott2016constrained} and is required in the complexity reduction algorithm.

\begin{proposition}\label{prop:HZ_reduct}
Let $\mathcal Z_h = HZ\langle \cc, \G^{c}, \G^{b}, \A^{c}, \A^{b}, \bb\rangle$. The set
$
\tilde{\mathcal{Z}}_h \triangleq  HZ\langle \cc+\Lambda_{G}\bb, \G^{c}-\Lambda_{G} \A^c, \G^{b}-\Lambda_{G}\A^b, \A^{c}-\Lambda_{A}\A^c, \A^{b}-\Lambda_{A} \A^b, \bb-\Lambda_{A}\bb\rangle
$
satisfies $\mathcal Z_h  \subseteq \tilde{\mathcal Z }_h$ for every $\Lambda_{G} \in \mathbb{R}^{n \times n_{c}}$ and $\Lambda_{A} \in \mathbb{R}^{n_{c} \times n_{c}}$.
\end{proposition}

\begin{proof} Clearly, 
$\bm z \in \mathcal{Z}_h$ if and only if $\exists \bm{\xi}^c \in B_{\infty}^{n_g}$ and $\exists \bm{\xi}^b \in \{-1,1\}^{n_b}$ such that $\mat{\bm z\\ \bm{0}}=\mat{\G^c\\\A^c} \bm{\xi}^c+\mat{\G^b\\\A^b} \bm{\xi}^b+\mat{\cc\\-\bb}$. For any such $\bm{\xi}^c$ and $\bm{\xi}^b$, $\mat{\bm z\\ \bm{0}}=\mat{\G^c\\\A^c} \bm{\xi}^c+\mat{\G^b\\\A^b} \bm{\xi}^b+\mat{\cc\\-\bb} + \mat{\Lambda_G(\bb-\A^c\bm{\xi}^c-\A^b\bm{\xi}^b)\\\Lambda_A(\bb-\A^c\bm{\xi}^c-\A^b\bm{\xi}^b)}$, which implies $\bm z\in \tilde{\mathcal{Z}}_h$.
\end{proof}

Proposition \ref{prop:HZ_reduct} allows us to choose any $\Lambda_{A}$ and $\Lambda_{G}$ to get an over-approximation of a hybrid zonotope. Next, we will introduce a heuristic approach to select proper $\Lambda_{A}$ and $\Lambda_{G}$ that leads to a less conservative over-approximation.

Consider the equality constraints of the hybrid zonotope $\A^c \bm{\xi}^c+\A^b \bm{\xi}^b = \bb$, which can be represented equivalently as
\begin{align}
     \sum_{j\in\{1,\dots,n_g\}} \A^c[i,j] \bm{\xi}^c[j] + &\A^b[i,j] \bm{\xi}^b[j] = \bb[i], \label{equ:linear_const}\\ &\forall i\in\{1,\dots,n_c\}.\notag
\end{align}

Following the procedure in \cite{scott2016constrained}, choose 
\begin{equation}\label{equ:lambda}
    \Lambda_G \!=\! \G^c \mathbf{E}_{c,r} (\A^c[r,c])^{-1},\Lambda_A \!=\! \A^c \mathbf{E}_{c,r} (\A^c[r,c])^{-1},
\end{equation}
where $\mathbf{E}_{c,r} \in \mathbb{R}^{n_g\times n_c}$ is zero except for a one in the $(c,r)$ position and $\A^c[r,c]$ is the entry of $\A^c$ in the $(r,c)$ position. With $
\tilde{\mathcal{Z}}_h =  HZ\langle \tilde{\cc}, \tilde{\G}^{c}, \tilde{\G}^{b}, \tilde{\A}^{c}, \tilde{\A}^{b}, \tilde{\bb}\rangle =  HZ\langle \cc+\Lambda_{G}\bb, \G^{c}-\Lambda_{G} \A^c, \G^{b}-\Lambda_{G}\A^b, \A^{c}-\Lambda_{A}\A^c, \A^{b}-\Lambda_{A} \A^b, \bb-\Lambda_{A}\bb\rangle
$, this transformation uses the $r$-th row of \eqref{equ:linear_const} to solve for $\bm{\xi}^c[c]$ in terms of $\bm{\xi}^c[k], k\in\{1,\dots,c-1,c+1,\dots,n_g\}$. This yields that $\tilde{\G}^{c}$ and $\tilde{\A}^{c}$ have identical zero $c$-th columns and $\tilde{\A}^{c}$, $\tilde{\A}^{b}$ and $\tilde{\bb}$ have identically zero $r$-th rows. Removing these columns and rows results in a hybrid zonotope with one less continuous generator and one less equality constraint.

This strategy ensures that the removed $r$-th equality constraint is still imposed in the reduced hybrid zonotope but the ability to constraint the $c$-th continuous variable is lost, i.e., $|\bm{\xi}^c[c]| \leq 1$. In order to select which continuous variable to eliminate, we consider the Hausdorff error introduced by reduction $d_H(r,c,\Zc_h) = \max_{\tilde{\z}\in\tilde{\Zc}_h}\min_{\z\in\Zc_h} ||\tilde{\z}-\z||_2$.

\subsection{Reduction Algorithm}

Algorithm \ref{alg:red} summarizes the procedures in this section for reducing the numbers of continuous and binary generators of hybrid zonotopes computed by the reachability analysis \eqref{equ:reach-cl}. In this algorithm, Line 1-5 perform binary generator reduction, Line 6-7 are used to remove redundant parallel continuous generators, and Line 8-11 implement Proposition \ref{prop:HZ_reduct} to further reduce continuous generators.

\begin{algorithm}[t]
\caption{Complexity reduction for hybrid zonotopes computed by \eqref{equ:reach-cl}}\label{alg:red}
\KwIn{hybrid zonotope $\mathcal{Z}_h=HZ\langle \cc, \G^{c}, \G^{b},\A^{c},$ $ \A^{b}, \bb\rangle$, $\hat{n}_g$ - number of continuous generators to reduce, $\hat{n}_b$ - number of continuous generators to reduce}
\KwOut{reduced hybrid zonotope $\hat{\mathcal{Z}}_h$}
 $\mat{\G^b_1 & \G^b_2} \longleftarrow \G^b$\tcp*{Partition by $\hat{n}_b$}
 $\mat{\A^b_1 & \A^b_2} \longleftarrow \A^b$\tcp*{Partition by $\hat{n}_b$}
 $\hat{\G}^c\longleftarrow \mat{\G^c & \G^b_1}$, $\hat{\A}^c\longleftarrow \mat{\A^c & \A^b_1}$\\
 $\hat{\G}^b \longleftarrow \G^b_2$, $\hat{\A}^b \longleftarrow \A^b_2$, $\hat\cc \longleftarrow \cc$, $\hat\bb \longleftarrow \bb$\\
 $\hat{\mathcal{Z}}_h \longleftarrow HZ\langle \hat\cc, \hat\G^{c}, \hat\G^{b},\hat\A^{c},$ $ \hat\A^{b}, \hat\bb\rangle$\\
 $\hat{\mathcal{Z}}_h^+ \longleftarrow $ lift $\hat{\mathcal{Z}}_h$ using \eqref{equ:lift}\\
 $\hat{\mathcal{Z}}_h \longleftarrow$ remove parallel generators in $\hat{\mathcal{Z}}_h^+$ and unlift\\
 \For{$i\in\{1,\dots,\max\{n_c,\hat{n}_g\}\}$}{
 $(r,c) \longleftarrow \text{argmin}_{r,c} d_H(r,c,\hat{\mathcal{Z}}_h)$\\
 $(\Lambda_G, \Lambda_A) \longleftarrow $ \eqref{equ:lambda}\\
 $\hat{\mathcal{Z}}_h \longleftarrow HZ \langle \hat\cc+\Lambda_{G}\hat\bb, \hat\G^{c}-\Lambda_{G} \hat\A^c, \hat\G^{b}-\Lambda_{G}\hat\A^b, \hat\A^{c}-\Lambda_{A}\hat\A^c, \hat\A^{b}-\Lambda_{A} \hat\A^b, \hat\bb-\Lambda_{A}\hat\bb \rangle$
 }
\KwRet $\hat{\mathcal{Z}}_h$
\end{algorithm}

\section{Simulation}\label{sec:sim}
In this section, two simulation examples are provided to demonstrate the performance of the proposed hybrid zonotope-based reachability analysis method. 


\begin{example}\label{exp:double}

Consider a double integrator model \cite{hu2020reach,everett2021reachability}:
$
{\x}{(t+1)}=\left[\begin{array}{ll}
1 & 1 \\
0 & 1
\end{array}\right] {\x}(t)+\left[\begin{array}{c}
0.5 \\
1
\end{array}\right] {\u}(t).
$
We use the same 3-layer FNN with ReLU activation functions in \cite{everett2021reachability} as the feedback controller. Algorithm \ref{alg:exact} is implemented to get exact output sets of the FNN and then utilized to compute the reachable sets of the closed-loop system for $T = 2$ time steps based on Theorem \ref{thm:sum} and Algorithm \ref{alg:G}. The initial set is given by 
$\mathcal{X}_0 = HZ\left\langle \mat{2.5\\0},\mat{0.2 & 0\\0 & 0.2},\mat{0.25\\0},\emptyset,\emptyset,\emptyset \right\rangle$.

We denote the proposed exact reachability analysis method based on \eqref{equ:reach-cl} and Theorem \ref{thm:sum} as Reach-HZ. We compare the proposed method with the Reach-CZ algorithm (\cite{zhang2022safety}), the Reach-LP algorithm (\cite{everett2021reachability}), and the Reach-SDP algorithm (\cite{hu2020reach}). For the latter two algorithms, we also test the version with initial set partition, i.e., Reach-LP-Partition and Reach-SDP-Partition. Table \ref{tbl:1} summarizes the computation times and set over-approximation errors for the proposed method and other state-of-the-art methods. The approximation errors are computed based on the difference ratio of sizes of computed reachable sets and exact reachable sets at the last time step. Note that although both the proposed Reach-HZ method and the Reach-CZ method can return the exact reachable sets, Reach-HZ only takes about half of the time of Reach-CZ. This results from the fact that Reach-CZ computes each reachable set as multiple constrained zonotopes while our Reach-HZ represents each reachable set compactly as a single hybrid zonotope. Our algorithms are implemented in Python with Gurobi \cite{gurobi}. The computer used for all the algorithms has a 3.7GHz CPU and 32GB memory.

Figure \ref{fig:exp2} illustrates reachable sets of the double integrator system using different methods. It can be observed that both our method and Reach-CZ provide more accurate reachable sets for all the time steps compared with other methods. For the safety verification, we consider a star-shaped unsafe region as plotted in Figure \ref{fig:exp2}, which is represented by a hybrid zonotope. Table \ref{tbl:2} compares the times used to solve the safety verification conditions between the MILP-based method in Proposition \ref{prop:verify} and the LP-based method in \cite{zhang2022safety}. 
We use the commercial solver Gurobi to solve the MILP-based conditions and the runtime is shorter than the runtime of MOSEK \cite{mosek} solving the LP-based conditions. 
This is due to the fact that the LP-based conditions require solving multiple LPs for each time step, while in our approach, only one MILP is solved for each time step.
\end{example}

\begin{table}[!ht]
\centering
\scalebox{1.1}{
\begin{tabular}{|c|c|c|}
\hline
Algorithm                  & Runtime [s] & Approx. Error \\ \hline
Reach-HZ (ours) &  0.146  &  0 \\ \hline
Reach-CZ \cite{zhang2022safety}           & 0.312 &  0  \\ 
Reach-LP \cite{everett2021reachability}                   &     0.032    &       3.34             \\ 
Reach-LP-Partition         &   2.297     &       0.23          \\ 
Reach-SDP \cite{hu2020reach}       &   108.77   &    0.79   \\ 
Reach-SDP-Partition        &   5222.91  &    0.33 \\ \hline
\end{tabular}}
\caption{Comparison of different reachability methods for Example \ref{exp:double}. Reach-HZ returns exact reachable sets within a shorter time compared with Reach-CZ.}
\label{tbl:1}
\vspace{-0.4cm}
\end{table}

\begin{table}[!ht]
\centering
\scalebox{1.1}{
\begin{tabular}{|c|c|c|}
\hline
 Problem Formulation & LP \cite{zhang2022safety} & MILP (ours) \\ \hline
Runtime [s] & 0.063 & 0.016 \\ \hline
\end{tabular}}
\caption{Runtimes of the LP-based method in \cite{zhang2022safety} and the MILP-based condition in Proposition \ref{prop:verify}.}
\label{tbl:2}\vspace{-0.4cm}
\end{table}

\begin{figure}[!ht]
    \centering
        \includegraphics[width=0.37\textwidth]{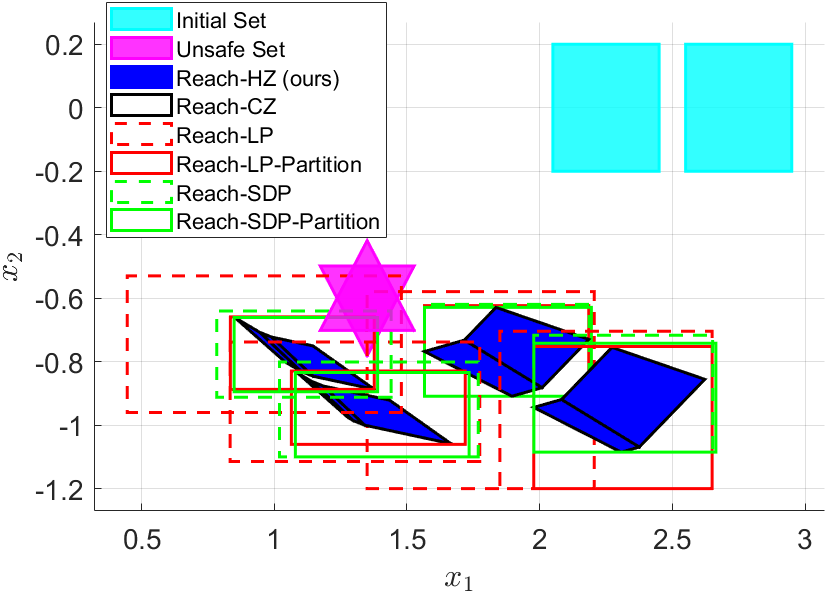}
     \caption{Reachable sets computed for the double integrator example. The initial set $\mathcal{X}_0$ is shown in cyan and the star-shaped unsafe region is in magenta. Reachable set computed by Reach-HZ is plotted in blue. Reachable set computed by Reach-CZ (\cite{zhang2022safety}) is in black, the LP-based method (\cite{everett2021reachability}) is in red, and the SDP-based method (\cite{hu2020reach}) is in green.}
      \label{fig:exp2}
\end{figure}

\begin{example}\label{exp:high-dim}
Consider a 4-D lateral dynamics model:
$$
\bm{x}(t+1) = \mat{0 & 1 & 5&0\\0&-5&0&-9.5\\0&0&0&1\\0&0.05&0&-2.8} \bm{x}(t) + \mat{0\\ 25 \\0\\50} \bm{u}(t).
$$
A 2-layer FNN is employed as the feedback controller and the initial set is given by $\mathcal{X}_0 = [0.1,0.9]\times[-0.9,-0.1]\times[0.05,0.15]\times[0.05,0.15]$.

We implement the proposed Reach-HZ method and the relaxed version denoted as Reach-HZ-Relax which combine Reach-HZ with the complexity reduction method in Algorithm \ref{alg:red}. We also run the Reach-CZ and Reach-CZ-Approx algorithms from \cite{zhang2022safety} for comparison. Figure \ref{fig:exp3} shows the one-step reachable sets of the lateral dynamics system computed by four different methods and Table \ref{tbl:3} summarizes their runtimes. Similar to the previous example, our Reach-HZ method provides the exact reachable set in a shorter time compared with Reach-CZ. Furthermore, our Reach-HZ-Relax method provides the tightest convex relaxation (the convex hull) of the reachable set while Reach-CZ-Approx returns a more conservative convex relaxation.
\end{example}
\begin{figure}[t]
    \centering
        \includegraphics[width=0.35\textwidth]{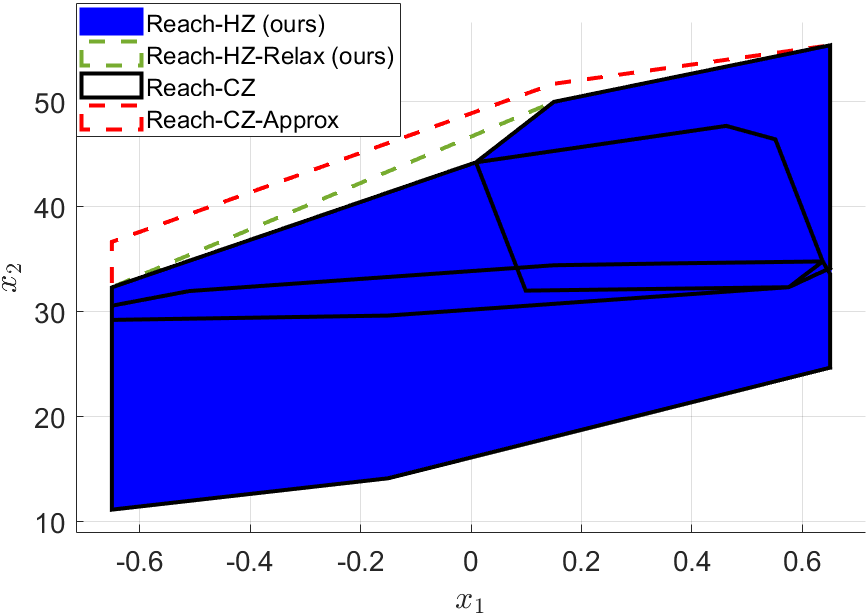}
     \caption{Reachable sets computed for the lateral dynamics example. Reachable set computed by Reach-HZ is plotted in blue and the Reach-CZ method (\cite{zhang2022safety}) is in black. Our Reach-HZ-Relax computes the tightest convex relaxation (green), while the Reach-CZ-Approx method (\cite{zhang2022safety}) only provides a more conservative convex relaxation (red).}
      \label{fig:exp3}
\end{figure}
\begin{table}[t]
\centering
\scalebox{1.1}{
\begin{tabular}{|c|c|c|}
\hline
Algorithm   & Runtime [s] & Approx. Error \\ \hline
Reach-HZ (ours) &  0.062  &  0 \\ 
Reach-HZ-Relax (ours) &  0.079  & 0.08 \\ \hline
Reach-CZ \cite{zhang2022safety}           & 0.110 &  0  \\ 
Reach-CZ-Approx \cite{zhang2022safety}                   &     0.047    &       0.20            \\ \hline
\end{tabular}}
\caption{Comparison of different reachability-based methods for the lateral dynamics example.}
\label{tbl:3}
\vspace{-0.3cm}
\end{table}

\section{Conclusion}\label{sec:concl}

In this work, we introduce a novel approach for computing exact reachable sets for neural feedback systems based on hybrid zonotopes. We show that when the input set is a hybrid zonotope, the computed reachable sets can also be compactly represented by hybrid zonotopes. 
Based on the reachability analysis, an MILP-based condition is presented for safety verification of the neural feedback system. Complexity reduction techniques are also proposed for the hybrid zonotopes to reduce the computation burden. As demonstrated in two numerical examples, the proposed approach outperforms other methods for reachability analysis and safety verification of neural feedback systems.

\bibliographystyle{IEEEtran}
\bibliography{ref}

\end{document}